\documentclass[11pt]{amsart}

\usepackage{amsfonts}
\usepackage{amssymb, latexsym, amsthm, verbatim,stmaryrd}
\usepackage{indentfirst}
\usepackage{amsmath}

\allowdisplaybreaks

\numberwithin{equation}{section}
\theoremstyle{definition}
\newtheorem{Thm}[equation]{Theorem}
\newtheorem{Prop}[equation]{Proposition}
\newtheorem{Cor}[equation]{Corollary}
\newtheorem{Lem}[equation]{Lemma}
\newtheorem{Def}[equation]{Definition}
\newtheorem{Exa}[equation]{Example}
\newtheorem{Rmk}[equation]{Remark}
\newtheorem*{Que*}{Question}

\makeatletter
\def\imod#1{\allowbreak\mkern5mu{\operator@font mod}\,\,#1}
\makeatother

\begin{document}

\title[Isomorphism, Zagier Duality, and Integrality]{Zagier Duality and Integrality of Fourier Coefficients for Weakly Holomorphic Modular Forms}
\author[Yichao Zhang]{Yichao Zhang}
\address{Department of Mathematics, University of Connecticut, Storrs, CT 06269}
\email{yichao.zhang@uconn.edu}
\date{}

\begin{abstract}
In this note, we generalize the isomorphism in \cite{zhang2014isomorphism} between vector-valued modular forms and scalar-valued modular forms to the case when the discriminant form is not necessarily induced from real quadratic fields. In particular, this general setting includes all of the subspaces with $\epsilon$-conditions of weakly holomorphic modular forms of integral weight, only two special cases of which were treated in \cite{zhang2014isomorphism}. With this established, we shall prove the Zagier duality for the canonical bases. Finally we raise a question on the integrality of the Fourier coefficients of these bases elements, or equivalently we concern the existence of a Miller-like basis for vector-valued modular forms.
\end{abstract}

\maketitle

\section*{Introduction}

\noindent
Let $D$ be a discriminant form of even signature $r$ and let $k\in\mathbb Z$ such that $k\equiv \frac{r}{2}\imod 2$. The discriminant form $D$ determines the Weil representation $\rho_D$ of $\text{SL}_2(\mathbb Z)$ on $\mathbb C[D]$, the group algebra of $D$. Let $N$ be the level of $D$ and $\{e_\gamma\colon \gamma\in D\}$ be the standard basis for $\mathbb C[D]$. Then it is known that there is a uniquely determined Dirichlet character $\chi_D$ modulo $N$ such that $\rho_D(M)e_0=\chi_D(M)e_0$ for each $M\in\Gamma_0(N)$. This enables us to construct a map $\phi$ from vector-valued modular forms of type $\rho_D$ to scalar-valued modular forms of level $N$ and character $\chi_D$. In the other direction, by averaging over the cosets of $\Gamma_0(N)$ in $\text{SL}_2(\mathbb Z)$, a scalar-valued modular form can be sent to a vector-valued modular form; we denote such a map by $\psi$. For the precise definition of these maps, see Definition 3.2.

Some properties of these maps are known. By explicitly working out the formulas for the Weil representation $\rho_D$, Scheithauer \cite{scheithauer2011some} proved the surjectivity of $\psi$ when $N$ is square-free. When $D$ has an odd prime discriminant, by introducing the plus and minus subspaces, Bruinier and Bundschuh \cite{bruinier2003borcherds} proved that the above maps are actually isomorphisms. We generalized their results to the case when the discriminant form is given by the norm form of a real quadratic field by introducing subspaces for sign vectors $\epsilon$ in \cite{zhang2014isomorphism}. In particular, the level $N$ can be any fundamental discriminant of a real quadratic field. The corresponding results, such as rationality of Fourier coefficients and the obstruction theorem for weakly holomorphic modular forms with prescribed principal parts, were also derived to this more general case. Those weakly holomorphic modular forms are holomorphic on the upper half plane but may possess poles at the cusps; they play an important role in Borcherds's theory of automorphic products (See, for example, \cite{borcherds1998automorphic}). The obstruction theorem, proved by Borcherds \cite{borcherds1999gross} for vector-valued modular forms and translated to the case of scalar-valued modular forms when $N$ is a prime by Bruinier and Bundschuh  \cite{bruinier2003borcherds}, states precisely when a polynomial in $q^{-1}$ can be a principal part of some weakly holomorphic modular forms.
In the paper \cite{zhang2014isomorphism}, we considered only the sign vectors that are relevant to Borcherds's automorphic products.

The first goal of this paper is to generalize the results in \cite{zhang2014isomorphism} to all of the sign vectors. More precisely, we shall treat the case when $N$ is the conductor of a primitive quadratic Dirichlet character $\chi$. For example, $N$ can be odd and squarefree, or the fundamental discriminant for a real quadratic field. Therefore, we cover more cases on the level $N$ and, as we shall see, all $\epsilon$-subspaces by varying $D$. See Theorem 3.3 for the isomorphisms and Theorem 4.4 for the obstruction theorem. In this paper, $N$ is the level of $D$ or just $N=|D|$, and one may also construct $D$ from $N$ and a sign vector $\epsilon$. See Remark 4.2 for details.

The second part of this paper is concerned with Zagier duality. In \cite{zagier1998traces}, Zagier proved a duality, known as Zagier duality, between the Fourier coefficients of certain weakly holomorphic modular forms of weight $1/2$ and the Fourier coefficients of certain weakly holomorphic modular forms of weight $3/2$.

Zagier duality has been discovered in many other cases ever since.  See \cite{kim2006traces}, \cite{bringmann2007arithmetic}, \cite{folsom2008duality} and \cite{bringmann2010zagier} in the half integral weight case. 
In the case of integral weights, Cho and Choie \cite{cho2011zagier} proved such duality between vector-valued harmonic weak Maass forms and vector-valued weakly holomorphic modular forms, generalizing Guerzhoy's result \cite{guerzhoy2009weak} for the full level case where Guerzhoy called these dual pairs \emph{grids}. With the nice isomorphism between vector-valued and scalar-valued modular forms proved by Bruinier and Bundschuh in \cite{bruinier2003borcherds}, Rouse \cite{rouse2006zagier} found the duality between weakly holomorphic modular forms of level $p=5,13,17$. His argument involves the explicit decomposition into the plus and the minus subspaces and the explicit action of Hecke operators. Later Choi \cite{choi2006simple} gave a simpler proof of this duality, where everything boils down to the well-known residue theorem on compact Riemann surfaces. Choi's result did not generalized Rouse's result because the existence forces $p=5,13,17$. Interestingly, Duke and Jenkins \cite{duke2008zeros} obtained Zagier duality for level $1$ weakly holomorphic modular forms by finding a Miller-like canonical basis and then a double variable generating series. Choi and Kim \cite{choi2013basis} generalized their results to the case of prime level $p$ such that the genus of the $\Gamma_0^+(p)$ is 0.

With our isomorphisms and the obstruction theorem established in the first part, we will introduce the notion of \emph{reduced modular forms}. These modular forms form a basis for the whole space of weakly holomorphic modular forms with some $\epsilon$-condition and we call it the canonical basis. We then prove the Zagier duality for such canonical bases (Theorem 5.7); namely such canonical bases and their dual bases make up the grids for the Zagier duality. In other words, we will resolve the issue of non-trivial obstructions, remove the dependence on existence, and obtain the complete Zagier duality.

Finally, we consider the integrality of Fourier coefficients $a(n)$ of our bases of reduced modular forms. This result plays a crucial role in \cite{kim2012rank} and \cite{kim2013weakly}, where the coefficients $s(n)a(n)$ represent the multiplicity of roots in some generalized Kac-Moody superalgebras. Here $s(n)$ appears in our isomorphism and is defined to be $2^{\omega((n,N))}$ (Section 1). It is well-known that there exists the Miller basis $\{f_1,f_2,...,f_d\}$ for the space of level $1$ cusp forms (see, for example, Chapter X, Theorem 4.4 in \cite{lang2001introduction}). More precisely, all Fourier coefficients are integral and the first $d$ coefficients of the basis give us the identity matrix. Such a basis can be extended to the space of holomorphic modular forms easily, and further to that of weakly holomorphic modular forms (\cite{duke2008zeros}).

In the case of higher level, cusps other than $\infty$ appear. We prove in this paper that the holomorphy of a modular form with some $\epsilon$-condition at $\infty$ dominates the holomorphy at any other cusp (Proposition 3.4 and Corollary 3.6), so it is natural to consider these higher level modular forms. We shall consider the integrality $s(n)a(n)\in\mathbb Z$ for reduced modular forms $f=\sum_na(n)q^n$, and because of the isomorphism, such integrality for the canonical basis essentially concerns the existence of a Miller-like basis for vector-valued modular forms.
We do not know how to prove such integrality systematically for all $N$, but we can reduce the problem to testing a finite number of reduced modular forms for each fixed $N$ and weight $k$, hence to computational verification by Sturm's theorem.

Here is the layout of this paper according to sections. In Section 1, we provide definitions and fix notations for discriminant forms and modular forms.
In Section 2, we fix one discriminant form and the corresponding sign vector $\epsilon$, and we also reproduce some results in \cite{zhang2014isomorphism} to our more general setting.
In Section 3, we establish the isomorphism and consider the behavior of a modular form with $\epsilon$-condition at other cusps.
From Section 4 on, we shall vary our discriminant form $D$ and hence the $\epsilon$-condition.  We prove the obstruction theorem and consider the rationality of Fourier coefficients in Section 4.
In Section 5, we introduce reduced modular forms and prove the Zagier duality. Some examples are presented. Finally, we propose the problem on the integrality of Fourier coefficients of reduced modular forms and provide some idea and one example on how to solve this computationally in the last section.

\subsection*{Acknowledgments} We thank Professor Henry H. Kim and Professor Kyu-Hwan Lee for very carefully reading a previous version of this paper and making many useful comments. We also thank Professor Kathrin Bringmann for bringing her paper with Olav Richter to my attention. 

\section{Preliminaries}\label{Preliminaries}
\noindent
We recall some definitions on discriminant forms and modular forms, and fix some notations. For more details on discriminant forms, one may consult \cite{conway1998sphere}, \cite{nikulin1980integral}, or \cite{scheithauer2009weil}.

A discriminant form is a finite abelian group $D$ with a quadratic form $q: D\rightarrow \mathbb Q/\mathbb Z$, such that the symmetric bilinear form defined by $(\beta,\gamma)=q(\beta+\gamma)-q(\beta)-q(\gamma)$ is nondegenerate, namely, the map $D\rightarrow \text{Hom}(D,\mathbb Q/\mathbb Z)$ defined by $\gamma\mapsto (\gamma,\cdot)$ is an isomorphism. We shall also write $q(\gamma)=\frac{\gamma^2}{2}$. We define the level of a discriminant form $D$ to be the smallest positive integer $N$ such that $Nq(\gamma)=0$ for each $\gamma\in D$. It is well-known that if $L$ is an even lattice then $L'/L$ is a discriminant form, where $L'$ is the dual lattice of $L$. Conversely, any discriminant form can be obtained this way. With this, we define the signature $\text{sign}(D)\in \mathbb Z/8\mathbb Z$ to be the signature of $L$ modulo $8$ for any even lattice $L$ such that $L'/L=D$.

Every discriminant form can be decomposed into a direct sum of Jordan $p$-components for primes $p$ and each Jordan $p$-component can be written as a direct sum of indecomposible Jordan $q$-components with $q$ powers of $p$. Such decompositions are not unique in general. To fix our notations, we recall the possible indecomposible Jordan $q$-components as follows.

Let $p$ be an odd prime and $q>1$ be a power of $p$. The indecomposible Jordan components with exponent $q$ are denoted by $q^{\delta_q}$ with $\delta_q=\pm 1$; it is a cyclic group of order $q$ with a generator $\gamma$, such that $q(\gamma)=\frac{a}{q}$ and $\delta_q=\left(\frac{2a}{p}\right)$. These discriminant forms both have level $q$.

If $q>1$ is a power of $2$, there are also precisely two indecomposable \emph{even} Jordan components of exponent $q$, denoted $q^{\delta_q 2}$ with $\delta_q=\pm 1$; it is a direct sum of two cyclic groups of order $q$, generated by two generators $\gamma$, $\gamma'$, such that if $\delta_q=1$, we have
\[q(\gamma)=q(\gamma')=0, \quad (\gamma,\gamma')=\frac{1}{q},\] and if $\delta_q=-1$, we have
\[q(\gamma)=q(\gamma')=\frac{1}{q},\quad (\gamma,\gamma')=\frac{1}{q}.\]
Such components have level $q$. There are also \emph{odd} indecomposable Jordan components in this case, denoted by $q_t^{\pm 1}$ with $\pm 1=\left(\frac{2}{t}\right)$ for each $t\in \left(\mathbb Z/8\mathbb Z\right)^\times$. Explicitly, $q^{\pm 1}_t$ is a cyclic group of order $q$ with a generator $\gamma$ such that $q(\gamma)=\frac{t}{2q}$. Clearly, these discriminant forms have level $2q$.

To give a finite direct sum of indecomposable Jordan components of the same exponent $q$, we multiply the signs, add the ranks, and add all subscripts $t$ ($t=0$ if there is no subscript). So in general, the $q$-component of a discriminant form is given by $q_t^{\delta_q n}$ ($t=0$ if $q$ is odd or the form is even). Set $k=k(q^{\delta_q n}_t)=1$ if $q$ is not a square and $\delta_q=-1$, and $0$ otherwise. If $q$ is odd, then define $p$-excess$(q^{\pm n})=n(q-1)+4k\imod 8$, and if $q$ is even, then define oddity$(q^{\pm n}_t)=2$-excess$(q^{\pm n}_t)=t+4k\imod 8$.

Let $D$ be a discriminant form and assume that $D$ has a Jordan decomposition $D=\oplus_q q^{\delta_q n_q}_t$ where the sum is over distinct prime powers $q$. Then
\[p\text{-excess}(D)=\sum_{q: p\mid q} p\text{-excess}(q^{\delta_q n_q}_t).\]
We have the \emph{oddity formula}:
\[\text{sign}(D)+\sum_{p>2}p\text{-excess}(D)=\text{oddity}(D)\imod 8.\]

Throughout this note, $k$ will be an integer and $\mathbb H$ will denote the upper half plane.
Let $M\in\text{GL}_2^+(\mathbb R)$, a real square matrix of size two and of positive determinant, and $f$ be a function on $\mathbb H$. The weight-$k$ slash operator of $M$ is defined as
\[\left(f|_kM\right)(\tau)=(\text{det}(M))^{\frac{k}{2}}(c\tau+d)^{-k}f(M\tau), \quad M=\begin{pmatrix}a&b\\c&d\end{pmatrix},\]
where $\tau$ is the variable on $\mathbb H$ and $M\tau=(a\tau+b)(c\tau+d)^{-1}$.
In $\text{GL}_2^+(\mathbb R)$, we denote
\[
S=\begin{pmatrix}
0&-1\\
1&0
\end{pmatrix},\quad T=\begin{pmatrix}
1&1\\
0&1
\end{pmatrix},\quad I=\begin{pmatrix}
1&0\\
0&1
\end{pmatrix},\quad W(m)=\begin{pmatrix}
0&-1\\
m&0
\end{pmatrix},\quad V(m)=\begin{pmatrix}
m&0\\
0&1
\end{pmatrix},
\] for a positive integer $m$.
We know that $T, S$ are the standard generators for $\text{SL}_2(\mathbb Z)$. Given any discriminant form $D$, let $r$ denote the signature of $D$; we assume throughout this note that $r$ is even and $k\equiv \frac{r}{2}\imod 2$. Let $\{e_\gamma: \gamma\in D\}$ be the standard basis of the group algebra $\mathbb C[D]$. The Weil representation $\rho_D$ attached to $D$ is a unitary representation of $\text{SL}_2(\mathbb Z)$ on $\mathbb C[D]$ such that
\begin{eqnarray*}
\rho_D(T)e_\gamma&=&e(q(\gamma))e_\gamma,\\
\rho_D(S)e_\gamma&=&\frac{i^{-\frac{r}{2}}}{\sqrt{|D|}}\sum_{\beta\in D}e(-(\beta,\gamma))e_\beta,
\end{eqnarray*}
where $e(x)=e^{2\pi i x}$ and $|D|$ is the order of $D$. In particular, we have $\rho_D(-I)e_\gamma=(-1)^{r/2}e_{-\gamma}$. Denote by $\text{Aut}(D)$ the automorphism group of $D$, that is, the group of group automorphisms of $D$ that preserve the norm (or the quadratic form). The action of elements in $\text{Aut}(D)$ and that of $\rho_D$ commute on $\mathbb C[D]$. We caution here that our $\rho_D$ is the same as that in \cite{borcherds1998automorphic} and in \cite{bruinier2003borcherds}, but conjugate to the one used in \cite{scheithauer2009weil} and \cite{scheithauer2011some}.

We denote by $\mathcal A(k,\rho_D)$ the space of functions $F=\sum_{\gamma\in D}F_\gamma e_\gamma$ on $\mathbb H$, valued in $\mathbb C[D]$, such that
\begin{itemize}
\item $F|_kM:=\sum_\gamma F_\gamma |_kM e_\gamma=\rho_D(M) F$ for all $M\in\text{SL}_2(\mathbb Z)$,
\item $F$ is holomorphic on $\mathbb H$ and meromorphic at $\infty$; namely, for each $\gamma\in D$, $F_\gamma$ is holomorphic on $\mathbb H$ and has Fourier expansion at $\infty$ with at most finitely many negative power terms.
\end{itemize}
More explicitly, if $F=\sum_\gamma F_\gamma\in \mathcal A(k,\rho_D)$, then
\[F_\gamma(\tau)=\sum_{n\in q(\gamma)+\mathbb Z, n\gg -\infty}a(\gamma,n)q^n.\] Denote by $\mathcal M(k,\rho_D)$ and $\mathcal S(k,\rho_D)$ the subspace of holomorphic modular forms and the subspace of cusp forms, respectively.
We define $\mathcal A^{\text{inv}}(k,\rho_D)$ to be the subspace of functions that are invariant under $\text{Aut}(D)$. The assumption that $k\equiv \frac{r}{2}$ actually says $F_\gamma=F_{-\gamma}$ for $F\in\mathcal A(k,\rho_D)$, so it must be imposed if we would like to have $F\in\mathcal A^{\text{inv}}(k,\rho_D)$, since $\gamma\mapsto -\gamma$ defines an element in $\text{Aut}(D)$. Similarly, we define $\mathcal M^{\text{inv}}(k,\rho_D)$ and $\mathcal S^\text{inv}(k,\rho_D)$.

For each positive integer $N$, let $\Gamma_0(N)$ denote the congruence subgroup of $\text{SL}_2(\mathbb Z)$ whose elements have left lower entry divisible by $N$. For each Dirichlet character $\chi$ of modulo $N$, we denote by $A(N,k,\chi)$ the space of weakly holomorphic modular forms of level $N$, weight $k$ and character $\chi$; namely the space of holomorphic functions $f$ on $\mathbb H$ such that $f|_kM=\chi(M)f$ for each $M\in\Gamma_0(N)$ and $f$ is meromorphic at cusps. The subspace of holomorphic forms and that of cuspforms are denoted by $M(N,k,\chi)$ and $S(N,k,\chi)$ respectively. Clearly, for these modular form spaces to be non-zero, we need $\chi(-1)=(-1)^k$. In the next section, for fixed discriminant forms $D$, we shall see that $\rho_D$ determines a Dirichlet character $\chi_D$ and $\chi_D(-1)=(-1)^{\frac{r}{2}}$. So the conditions we impose on $k$ are consistent.

Let $f=\sum_{n}a(n)q^n\in A(N,k,\chi)$. Then $P(q^{-1})=\sum_{n<0}a(n)q^n$ is a polynomial without constant term in $q^{-1}$ and we call $P(q^{-1})$ the \emph{principal part} of $f$ (at $\infty$).

For any positive integer $m$, we denote by $\omega(m)$ the number of distinct prime divisors of $n$. For any pair $m,N$ of integers, we denote by $(m,N)$ the greatest common divisor of $m$ and $N$, which should not be confused with the bilinear form. If $N>0$ and $m>1$, we denote $N_m$ to be the $m$-part of $N$; that is, $N_m\mid N$ is positive, contains only primes that divide $m$, and $(N/N_m,m)=1$. If $p$ be a prime and $l$ a non-negative integer, we denote $p^l||N$ if $p^l\mid N$ but $p^{l+1}\nmid N$.

For a Dirichlet character $\chi$ modulo $N$, we shall denote its $p$-component by $\chi_p$, hence $\chi=\prod_{p\mid N}\chi_p$. For each positive divisor $m$ of $N$, we define $\chi_m=\prod_{p\mid m}\chi_p$ and $\chi_m'=\prod_{p\nmid m}\chi_p$.
Let $W(\chi)$ denote the Gauss sum of $\chi$ and we know that if $p>2$ and $\chi_p=\left(\frac{\cdot}{p}\right)$, then $W(\chi_p)=\varepsilon_p p^{\frac{1}{2}}$; here $\varepsilon_p=1$ if $p\equiv 1\imod 4$, and $i$ if $p\equiv 3\imod 4$. Similarly,
\begin{itemize}
\item if $\chi_2=\left(\frac{-1}{\cdot}\right)$, $W(\chi_2)=\varepsilon_2 4^{\frac{1}{2}}$ with $\varepsilon_2=i$;
\item if $\chi_2=\left(\frac{2}{\cdot}\right)$, $W(\chi_2)=\varepsilon_2 8^{\frac{1}{2}}$ with $\varepsilon_2=1$;
\item if $\chi_2=\left(\frac{-2}{\cdot}\right)$, $W(\chi_2)=\varepsilon_2 8^{\frac{1}{2}}$ with $\varepsilon_2=i$.
\end{itemize}
Here $\varepsilon_p$ is not to be confused with the sign vectors $\epsilon_p$ defined in the following section.

For integers $i,j$, we define $\delta_{i,j}=1$ if $i=j$, and $0$ otherwise.

\section{Discriminant Forms and $\epsilon$-Condition}

\noindent
In this section, we first fix a discriminant form and investigate its properties, and then define the $\epsilon$-condition on scalar-valued modular forms.

From now on and until the end of Section 3, we fix a discriminant form $D=\oplus_p D_p$ of the following form:
if $p>2$, then $D_p=p^{\delta_p}$ with $\delta_p\in\{\pm 1\}$;
$D_2$ is trivial, or $2^{+2}_t$ with $t\in\{\pm 2\}$, or $2_{t_1}^{+1}\oplus 4_{t_2}^{\delta_2}$ with $\delta_2=\left(\frac{2}{t_2}\right)$ and $t_1\in\{\pm 1\}$, $t_2\in\{\pm 1,\pm 3\}$. We denote the level of $D$ by $N$; note that $|D|=N$. Note that $D_2$ means something else in \cite{scheithauer2009weil}.

Now let $D^*=D[-1]=\oplus_pD_p^*$ be the discriminant form with the same group but with quadratic form $q^*=-q$. We call $D^*$ the dual of $D$. It is not hard to see that if $p>2$ then $D_p^*=p^{\delta_p^*}$ with $\delta_p^*=\left(\frac{-1}{p}\right)\delta_p$. If $D_2=2^{+2}_t$ , then $D_2^*=2^{+2}_{-t}$, and if $D_2=2_{t_1}^{+1}\oplus 4_{t_2}^{\delta_2}$ then $D_2^*=2_{-t_1}^{+1}\oplus 4_{-t_2}^{\delta_2}$. It is clear that $D^*$ has the same level $N$ as $D$ does.

For a modular form $F\in\mathcal A^\text{inv}(k,\rho_D)$, define $W$ the span of $F_\gamma$, $\gamma\in D$, and $W'$ the span of $F_0|M$, $M\in SL_2(\mathbb Z)$. Let $W_0$ be the subspace of $T$-invariant functions in $W$.

We reproduce a few lemmas but omit their proofs since the corresponding proofs in \cite{zhang2014isomorphism} can be carried over.  They correspond to Proposition 2.3, Lemma 3.1, Lemma 3.2 and Proposition 3.3 in \cite{zhang2014isomorphism}, respectively.

\begin{Lem}\label{Invariance}
If $\beta,\gamma\in D$ with $q(\beta)=q(\gamma)$, then there exists $\sigma\in\text{Aut}(D)$ such that $\sigma\beta=\gamma$.
\end{Lem}

\begin{Lem}\label{Term-Inclusion}
Let $S\subset D$. If $\sum_{\gamma\in S}F_\gamma\in W'$, then $F_\gamma\in W'$ for any $\gamma\in S$.
\end{Lem}

\begin{Lem}\label{T-invariance}
$W_0=\text{span}_\mathbb{C}\{F_0\}$. Actually, if $f=\sum_{\gamma\in D}a_\gamma F_\gamma\in W_0$, then $f=a_0F_0$.
\end{Lem}

\begin{Lem}\label{Uniqueness}
$W=W'$. In particular, if $F_0=0$, then $F=0$.
\end{Lem}

Define the primitive Dirichlet character $\chi=\chi_D=\prod_p\chi_p$ of modulus $N$ as follows: if $p>2$ and $p\mid N$, then $\chi_p=\left(\frac{\cdot}{p}\right)$; $\chi_2$ is trivial if $D_2$ is trivial, $\chi_2=\left(\frac{-4}{\cdot}\right)$ if $D_2=2^{+2}_t$, and $\chi_2=\left(\frac{-2a}{\cdot}\right)$ with $a=\left(\frac{-1}{t_1t_2}\right)$ if $D_2=2^{+1}_{t_1}\oplus 4_{t_2}^{\delta_2}$. Such a character is determined by the Weil representation associated to $D$, justifying the notation.

For each sign vector $\epsilon'=(\epsilon'_p)_{p\mid N}\in\{\pm 1\}^{\omega(N)}$, we define the subspace in $A(N,k,\chi_D)$
\[A^{\epsilon'}(N,k,\chi_D)=\left\{\left.f=\sum_n a(n)q^n\in A(N,k,\chi_D)\right|
a(n)=0 \text{ if } \chi_p(n)=-\epsilon'_p \text{ for some }p\mid N\right\},\]
and we know that $A(N,k,\chi_D)=\oplus_{\epsilon'}A^{\epsilon'}(N,k,\chi_D)$ (\cite[Proposition 3.10]{zhang2014isomorphism}) where $\epsilon'$ runs through the whole set $\{\pm 1\}^{\omega(N)}$. The following lemma is Corollary 3.12 in \cite{zhang2014isomorphism}. See next section for the meaning of these operators.

\begin{Lem} Assume $f\in A(N,k,\chi_D)$. Then
$f\in A^{\epsilon'}(N,k,\chi_D)$ if and only if \[f|_kU(N_p)\eta_p=\epsilon_p'\varepsilon_p\chi_p(-1)N_p^{\frac{k-1}{2}}f, \quad\text{ for each } p\mid N.\]
\end{Lem}

Among these subspaces, we specify one of them, $A^{\epsilon}(N,k,\chi)$, with
$\epsilon=(\epsilon_p)_{p\mid N}$ defined as follows: if $p>2$, then $\epsilon_p=\chi_p(2N/p)\delta_p$; if $D_2=2_t^{+2}$, $\epsilon_2=\chi_2(Nt/8)$, and if $D_2=2^{+1}_{t_1}\oplus 4_{t_2}^{\delta_2}$, we set $\epsilon_2=\chi_2(t_2N/8)$.
By the definition of $D^*$, it can be seen easily that the sign vector $\epsilon^*$ for $D^*$ is given by $\epsilon^*_p=\chi_p(-1)\epsilon_p$ for each $p\mid N$.

For our choice of $N$, it is well-known that the inequivalent cusps of $\Gamma_0(N)$ are represented precisely by $\frac{1}{m}$ with $m$ running over the positive divisors of $N$. In particular, $\frac{1}{N}\sim\infty$ and $1\sim 0$ as cusps for $\Gamma_0(N)$.

\section{Correspondence between Vector-Valued Modular Forms and Scalar-Valued Modular Forms}

\noindent
In this section, we will generalize the results in \cite{zhang2014isomorphism}. Actually, we show that the isomorphism in \cite{zhang2014isomorphism} and other results also hold for our general sign vector $\epsilon$. We shall be brief on proofs in this section and for more details please see \cite{zhang2014isomorphism}.

Before we establish the isomorphism between $\mathcal A^\text{inv}(k,\rho_D)$ and $A^\epsilon(N,k,\chi_D)$, we first recall some Hecke operators. For $m\mid N$, the Hecke operator $U(m)$ on $A(N,k,\chi_D)$ is defined as
\[(f|_kU(m))(\tau)=m^{\frac{k}{2}-1}\sum_{j\imod m}f\left|_k\begin{pmatrix}
1&j\\
0&m
\end{pmatrix}\right..\]
If $f=\sum_{n\in\mathbb Z}a(n)q^n$, then $f|_kU(m)=\sum_{n\in\mathbb Z}a(mn)q^n$.

We shall need the so-called $W$-operators; here we follow Miyake's notations (\cite{miyake2006modular}) and denote them by $\eta_m$.
For a positive divisor $m$ of $N$, choose $\gamma_m\in\text{SL}_2(\mathbb Z)$ such that
\[\gamma_m\equiv
\left\{
\begin{array}{cl}
S&\imod (N_m)^2\\
I &\imod (N/N_m)^2
\end{array}
\right.,
\] and define $\eta_m=\gamma_mV(N_m)$ and denote $\eta_m'=\eta_{N/N_m}$. Recall that $N_m$ means the $m$-part of $N$. For completeness, we copy the following lemma from Lemma 1.1 in \cite{zhang2014isomorphism}. See Section 1 for the meaning of other notations.

\begin{Lem}\label{Eta-operator} Let $f\in A(N,k,\chi_D)$ and $m,m_1,m_2$ be positive divisors of $N$.

(1) The action $f|_k\eta_m$ is independent of the choice of $\gamma_m$ and it defines an operator on $A(N,k,\chi_D)$.

(2) $f|_k\eta_N=f|_kW(N)$.

(3) If $(m_1,m_2)=1$, $f|_k\eta_{m_1m_2}=\chi_{m_2}(N_{m_1})f|_k\eta_{m_1}\eta_{m_2}$. In particular, $f|_k\eta_m\eta_m'=\chi_m'(N_m)f|_kW(N)$. Moreover, if $m=p_1p_2\cdots p_k$ is square-free, then $f|_k\eta_m=\prod_{i<j}\chi_{p_j}(N_{p_i})f|_k\eta_{p_1}\eta_{p_2}\cdots\eta_{p_k}$.

(4) $f|_k\eta^2_m=\chi_m(-1)\chi_m'(N_m)f$.

(5) If $(m_1,m_2)=1$, $f|_k\eta_{m_1}U({m_2})=\chi_{m_1}({m_2})f|_kU({m_2})\eta_{m_1}$.
\end{Lem}

From now on, we shall drop the weight in the notations of the operators if no confusion is possible. We construct the isomorphisms in the following definition.

\begin{Def}

Define a map $\phi: \mathcal A^\text{inv}(k,\rho_D)\rightarrow A^\epsilon(N,k,\chi_D)$ by
\[F\mapsto i^{\frac{r}{2}}2^{-\omega(N)}N^{-\frac{k-1}{2}}F_0|{W(N)}.\]
Conversely, we define $\psi: A^\epsilon(N,k,\chi_D)\rightarrow \mathcal A^\text{inv}(k,\rho_D)$ by
\[f\mapsto i^{\frac{r}{2}}N^{\frac{k-1}{2}}\sum_{M\in\Gamma_0(N)\backslash SL_2(\mathbb Z)}\left(f|W(N)|M\right)\rho_D(M^{-1})e_0.\]
\end{Def}

For each integer $n$ we define $s(n)=2^{\omega((n,N))}$; it depends on $N$. For example, if $N=12$, then $s(0)=s(6)=4$ and $s(2)=s(3)=2$.

\begin{Thm}
The maps $\phi$ and $\psi$ are inverse isomorphisms between $\mathcal A^\text{inv}(k,\rho_D)$ and  $A^\epsilon(N,k,\chi_D)$. Explicitly, if $f=\sum_na(n)q^n\in A^\epsilon(N,k,\chi_D)$ and $\psi(f)=F=\sum_\gamma F_\gamma e_\gamma$, then
\[F_\gamma(\tau)=s(Nq(\gamma))\sum_{n\equiv Nq(\gamma)\imod N\mathbb Z}a(n)q^{\frac{n}{N}}=\sum_{n\equiv Nq(\gamma)\imod N\mathbb Z}s(n)a(n)q^{\frac{n}{N}}.\]
\end{Thm}
\begin{proof}
Following the same lines as in \cite{zhang2014isomorphism}, we sketch the proof.

That $\psi$ is well-defined follows easily, of which the invariance follows from the fact that the actions of $\text{SL}_2(\mathbb Z)$ and $\text{Aut}(D)$ on $\mathbb C[D]$ commute. On the other hand, $\phi(F)$ belongs to $A(N,k,\chi)$ by Proposition 4.5 in \cite{scheithauer2009weil}. To see that $\phi(F)$ satisfies the $\epsilon$-condition, we observe that
\[F_0|W(N)=i^{\frac{r}{2}}N^{\frac{k-1}{2}}\sum_{\gamma\in D}F_\gamma(N\tau)=i^{\frac{r}{2}}N^{\frac{k-1}{2}}\sum_{n\in\mathbb Z}\left(\sum_{\gamma: q(\gamma)=\frac{n}{N}}a(\gamma,nN^{-1})\right)q^n:=\sum_{n\in\mathbb Z}a(n)q^n.\] Now it is easy to see that
$q$ represents $\frac{n}{N}$ if and only if $q_p$ represents $\frac{nN/N_p}{N_p}$ for each prime $p\mid N$. We may verify case by case that $a(n)=0$ if $\chi_p(n)=-\epsilon_p$ for some $p\mid N$. So $\phi$ is well-defined.

The same argument in Proposition 3.5 of \cite{zhang2014isomorphism} can be carried over to prove that $\psi\circ\phi=id$; note that we need Lemma 2.3. The proof of $\phi\circ\psi=id$ is similar to that of Proposition 3.15 in \cite{zhang2014isomorphism}, where explicit formulas in Theorem 4.7 of \cite{scheithauer2009weil} and Lemma 3.1 are needed.
In order to convince the reader that this is the case, we sketch the proof of $\phi\circ\psi=id$ when $D_2=2_{t_1}^{+1}\oplus 4_{t_2}^{\delta_2}$. We leave other cases to the reader.

For each cusp $s$, define
\[F_s=(-1)^{\frac{r}{2}}\sum_{M\in\Gamma_0(N)\backslash SL_2(\mathbb Z)\atop M\infty\sim s}\left(f|W(N)|M\right)\rho_D(M^{-1})e_0.\]
It suffices to prove that $\sum_s(F_s,e_0)|W(N)=2^{\omega(N)}f$. Denote by $m_1$ an odd positive divisor of $N$.

We first consider a cusp $s$ of the form $\frac{1}{N/m_1}$. Since the cosets in \[\{M\in\gamma_0(N)\backslash \text{SL}_2(\mathbb Z)\colon M\infty\sim s\}\] can be represented by $\{\gamma_{m_1}T^j\colon j\imod m_1\}$, we first note that
\[(F_s,e_0)|W(N)=\left((-1)^km_1^{1-\frac{k}{2}}f|W(N)\eta_{m_1}U(m_1)W(N)\right)(\rho_D(\gamma_{m_1}^{-1})e_0,e_0):=A\cdot B,\]
where $B=(\rho_D(\gamma_{m_1}^{-1})e_0,e_0)$ and $A$ is the product of other factors.
By repeatedly use of Lemma 3.1, one can show that
\[A=m_1^{\frac{1}{2}}\chi_{m_1}(-2)\prod_{p\mid m_1}\delta_p\varepsilon_p\cdot f.\]
By Theorem 4.7 in \cite{scheithauer2009weil} (note that his Weil representations are conjugate to ours), we see that
\[B=-m_1^{-\frac{1}{2}}\chi_2(-1)\chi_{m_1}(2)\left(\frac{-1}{t_1t_2}\right)\prod_{p\mid m_1}\delta_p\varepsilon_p.\]
Since $\prod_{p\mid m_1}\varepsilon_p^2=\chi_{m_1}(-1)$ and $\chi_2(-1)=-\left(\frac{-1}{t_1t_2}\right)$, we have $AB=f$. Similar computations show that $(F_s, e_0)|W(N)=f$ if $s$ is of the form $\frac{1}{8m_1}$.

For a cusp $s$ of the form $\frac{1}{2m_1}$ or $\frac{1}{4m_1}$. We claim that $(F_s,e_0)|W(N)=0$. To explain this, we employ the notations in \cite{scheithauer2009weil}. Let $M=\begin{pmatrix}a&b\\c&d\end{pmatrix}\in\text{SL}_2(\mathbb Z)$ such that $M\infty\sim s$. Then we have $2||c$ or $4||c$. Therefore, $x_c\neq 0$ and $0\not\in D^{c*}$, and hence $(F_s,e_0)|W(N)=0$.

Putting everything together, we see that $\sum_s(F_s,e_0)|W(N)=2^{\omega(N)}f$.
\end{proof}

We now investigate the behavior of a weakly holomorphic form that satisfies the $\epsilon$-condition at all cusps. 

\begin{Prop}\label{Fourier-Cusps}
Let $f=\sum_na(n)q^n\in A^\epsilon(N,k,\chi_D)$ and let $s$ be a cusp and $q_s$ be the local parameter at $s$. Fix any positive odd divisor $m_1$ of $N$. Then

(1) If $s\sim\frac{1}{m}$ with $m=m_1$ or $N_2m_1$, then the Fourier expansion of $f$ at $s$ contains precisely powers of the form $q_s^{n}$ with $a(nN/m)\neq 0$.

(2) If $4\mid N$ and $s\sim \frac{1}{2m_1}$, then the Fourier expansion of $f$ at $s$ contains at most powers of the form $q_s^{n/2}$ with $a(nN/2m_1)\neq 0$.

(3) If $8\mid N$ and $s\sim\frac{1}{4m_1}$, then the Fourier expansion of $f$ at $s$ contains at most powers of the form $q_s^{n/2}$ with $a(nN/4m_1)\neq 0$.
\end{Prop}
\begin{proof} 
We denote $f=(*)g$ if $f=cg$ for some $c\in\mathbb C^\times$.

For (1), let $m=m_1$ or $m=N_2m_1$ accordingly.
It is not hard to see that $\gamma_{N/m}\infty\sim s$, so it suffices to consider the Fourier expansion of $f|\gamma_{N/m}$. Since $f|U(N/m)\eta_{N/m}=(*)f$, we have
\[f|\gamma_{N/m}=f|\eta_{N/m}V(N/m)^{-1}=(*) f|U(N/m)V(N/m)^{-1}.\]
Since the width of the cusp $s$ is $N/m$, Part (1) follows.

For (2), let us deal with the case when $4||N$. Let $\beta_{2m_1}=\begin{pmatrix}
a&b\\
c&d
\end{pmatrix}\in\text{SL}_2(\mathbb Z)$ be a matrix that is congruent to
\[
\begin{pmatrix}
1&0\\
2&1
\end{pmatrix}\imod 4^2,\quad
\begin{pmatrix}
0&-1\\
1&0
\end{pmatrix}\imod (N/4m_1)^2,\quad \text{and}\quad
\begin{pmatrix}
1&0\\
0&1
\end{pmatrix}\imod m_1^2.
\]
Clearly, $\beta_{2m_1}\infty\sim s$. Let us pass to vector-valued modular forms. From the isomorphism $f\mapsto F$,  we see that
\[f|\beta_{2m_1}=(*)F_0|W(N)\beta_{2m_1}=(*)F_0\left|\begin{pmatrix}-c&-d\\Na&Nb\end{pmatrix}\right..\]
Since $(c,N)=2m_1$, we choose any integers $u,v$ such that
$\beta=\begin{pmatrix}-c/2m_1&u\\aN/2m_1&v\end{pmatrix}\in\text{SL}_2(\mathbb Z)$. We then have
\[f|\beta_{2m_1}=(*)F_0\left|\beta\begin{pmatrix}2m_1&-vd-uNb\\0&N/2m_1\end{pmatrix}\right..\]
By Theorem 4.7 in \cite{scheithauer2009weil}, we have
\[F_0|\beta=(*)\sum_{\gamma\in D^{c'*}}e(d'\gamma_{c'}^2/2)F_\gamma,\] here $c'=aN/2m_1$, $d'=v$, $D^{c'*}=\gamma_2+\gamma_2'+\oplus_{p\mid m_1}D_p$ with $\gamma_2,\gamma_2'$ the generators of $D_2$, and for the meaning of $\gamma_{c'}^2/2$ see \cite{scheithauer2009weil}. In particular, if $F_\gamma$ appears in $F_0|\beta$, we must have $q(\gamma)\in \frac{1}{2m_1}\mathbb Z$. Moreover, from the isomorphism, we see that such $F_\gamma$ contains only $q^{\frac{n}{N}}$ with $a(n)\neq 0$ and $\frac{N}{2m_1}\mid n$. Therefore, since the width at $s$ is $N/4m_1$, we have $f|\beta_{2m_1}$ contains at most terms of the form $q_s^{n/2}$ with $a(nN/2m_1)\neq 0$.

Part (3) and the case when $8\mid N$ for Part (2) follow in the same way, and we omit the details.
\end{proof}

\begin{Rmk} Part (1) of Proposition 3.4 can be easily made precise using the $\epsilon$-condition.
We can also made Part (2) precise using the argument in the proof of Corollary 3.13 in \cite{zhang2014isomorphism}. For example, if $4||N$ and we choose for any odd $m_1\mid N$
\[\alpha_{2m_1}=\eta_{2m_1}^{-1}\begin{pmatrix}
1 & -1/2\\
0 &1
\end{pmatrix}\eta_{2m_1}\eta_{N/4m_1},\]
then $\alpha_{2m_1}V(N/4m_1)^{-1}\in\text{SL}_2(\mathbb Z)$ and it sends $\infty$ to the cusp $\frac{1}{2m_1}$. Our $\alpha_{2m_1}$ here  differs from the one used in \cite{zhang2014isomorphism} by $V(N/4m_1)$. Messy but elementary computations give us that
\begin{align*}f|\alpha_{2m_1}&=-2^{1-\frac{k}{2}}(N/4m_1)^{\frac{1-k}{2}}\chi'_{2}(2)\chi_{m_1}(N/4m_1)\prod_{p\neq p'\mid \frac{N}{4m_1}}\chi_p(p')\prod_{p\mid \frac{N}{m_1}}\epsilon_p\varepsilon_p^{-1}\\
&\hspace{2cm}\times \left(\sum_n\chi_2(n)a(2n)q^n\right)\left|U(N/4m_1)
\begin{pmatrix}
1& 0\\
0& 2
\end{pmatrix}\right..\end{align*}

\end{Rmk}

\begin{Cor}\label{Holomorphy}
Let $f\in A^\epsilon(N,k,\chi_D)$.

(1) If $f$ is holomorphic (or vanishes, respectively) at $\infty$, then $f\in M^\epsilon(N,k,\chi_D)$ (or $S^\epsilon(N,k,\chi_D)$, respectively).

(2) If $f=q^{-d}+O(1)$ with $d$ a positive integer coprime to $N$, then $f$ is holomorphic at cusps other than $\infty$.

(3) If $k\leq 0$ and $f,g\in A^\epsilon(N,k,\chi_D)$ have the same principal part at $\infty$, then $f=g$.
\end{Cor}
\begin{proof}
The first part follows trivially from Proposition 3.4. The second part also follows from this proposition; actually, if $s\sim\frac{1}{m_1}$ with $m_1$ odd and $q_s^n$ with $n<0$ appears, then we must have $N/m\mid d$ which is absurd since $m<N$ and $(d,N)=1$, and the other cases follows similarly.

For (3), we note that $f-g\in A^\epsilon(N,k,\chi)$ is holomorphic at $\infty$, hence holomorphic at all cusps by Part (1). But $k\leq 0$, hence $f-g=0$. Alternatively, we may derive this directly from the isomorphism. Indeed, if $F=\psi(f-g)$, then $F_\gamma$ is holomorphic from the isomorphism and in particular $F_0$ is holomorphic, hence $0$. So $F=0$ by Lemma 2.4 and hence $f-g=0$.
\end{proof}

\section{Obstructions and Rationality of Fourier Coefficients}

\noindent
In this section, we translate Borcherds's theorem of obstructions to scalar-valued modular forms using the isomorphism in the previous section. In other words, we investigate the existence of weakly holomorphic modular forms with prescribed principal parts. At the end of this section, we shall also mention the rationality of Fourier coefficients of modular forms that satisfy some $\epsilon'$-condition. The arguments are essentially the same at that in Section 4 of \cite{zhang2014isomorphism}, so we will only mention some differences.

We shall vary $D$ by choosing different data for $\delta_p$ or $t,t_1,t_2$, and hence vary $\epsilon$ and other data, from now on.

Let $m$ be a positive integer. Recall that if $F=\sum_\gamma F_\gamma e_\gamma$ and $G=\sum_\gamma G_\gamma e_\gamma$ with $F_\gamma,G_\gamma\in\mathbb C(\!(q_m)\!)$, the field of Laurent series in $q_m=q^{\frac{1}{m}}$, we have the following pairing:
\[\langle F,G\rangle=\text{ the constant term of }\sum_\gamma F_\gamma G_\gamma.\] The following proposition holds for each $D$ and its Weil representation.
 This is Corollary 4.3 in \cite{zhang2014isomorphism}, which follows easily from Borcherds's Theorem 3.1 in \cite{borcherds1999gross}.

\begin{Prop}\label{Obstruction-Invariant}
Let $P$ be a $\mathbb C[D]$-valued polynomial in $q^{-1}_N$ that is invariant under $\text{Aut}(D)$. Then there exists $F\in\mathcal A^{\text{inv}}(k,\rho_D)$ such that $F-P$ vanishes at $q_N=0$,  if and only if $\langle P, G\rangle=0$ for each $G\in\mathcal M^\text{inv}(2-k,\rho_D^*)$.
\end{Prop}

We now fix $N$ such that $N_p=p$ if $p\mid N$ is odd and $N_2=1,4$ or $8$. Define $\chi_p=\left(\frac{\cdot}{p}\right)$ if $p\mid N$ is odd, and define $\chi_2$ to be $1$ if $2\nmid N$, $\left(\frac{-4}{\cdot}\right)$ if $2^2||N$, $\left(\frac{2}{\cdot}\right)$ if $N\equiv 8\imod 32$, and $\left(\frac{-2}{\cdot}\right)$ if $N\equiv 24\imod 32$. We define $\chi=\prod_p\chi_p$. In any case, $\chi$ is a primitive character modulo $N$.

\begin{Rmk}
By varying our discriminant form $D=\oplus_pD_p$, the isomorphism actually covers $A^\epsilon(N,k,\chi)$ for all $\epsilon$. More explicitly, for any sign vector $\epsilon$,
\begin{itemize}
\item if $2\nmid N$, then we choose $D_p=p^{\delta_p}$ with $\delta_p=\chi_p(2N/p)\epsilon_p$;
\item if $2^2||N$, then we choose the same $D_p$ for odd $p$ and choose $D_2=2_t^{+2}$ with $t\in\{\pm 2\}$ determined by $t=2\chi_2(N/4)\epsilon_2$;
\item if $2^3|N$, then we choose the same $D_p$ for odd $p$ and choose $D_2=2_{t_1}^{+1}\oplus 4^{\delta_2}_{t_2}$ with $t_2\in\{\pm 1,\pm 3\}$ determined by $\chi_2(t_2)=\chi_2(N/8)\epsilon_2$ and $t_1\in\{\pm 1\}$ determined by $\left(\frac{-4}{t_1t_2}\right)=\chi_2(3)$. Note that we have two possible $D_2$'s in this case, but they are isomorphic.
\end{itemize}
For example, here we cover the cases when $N=15$ and $N=20$ for all possible $\epsilon$. Note that the case when $N$ is an odd prime is already contained in \cite{bruinier2003borcherds}.
\end{Rmk}

Let us ssume that $k\leq 0$ and hence $2-k\geq 2$.
Let us denote by $E(N,2-k,\chi)$ the space of Eisenstein series of level $N$, weight $2-k$ and character $\chi$. It is well-known that $\text{dim}(E(N,2-k,\chi))=2^{\omega(N)}$, with a basis concretely given by $\{E_m: m\mid N, m=N_m\}$ where (see Theorem 4.5.2 and Theorem 4.6.2 in \cite{diamond2005first}) :
\[
E_m= \delta_{1,m}L\left(k-1,\chi\right)+2\sum_{n=1}^\infty\left(\sum_{d\mid n}\chi_m(n/d)\chi_m'(d)d^{1-k}\right)q^n.
\]
For each $\epsilon$ and $m\mid N$, we denote $\epsilon_m=\prod_{p\mid m}\epsilon_p$.
Define \[E^{\epsilon}=\frac{1}{s(0)L(k-1,\chi)}\sum_{m\mid N}\epsilon_m E_m,\] and assume $E^\epsilon=\sum_{n}B(n)q^n$. We see that $B(0)=s(0)^{-1}$ and such normalization is different from \cite{bruinier2003borcherds} or \cite{zhang2014isomorphism}.

\begin{Lem}\label{Eisenstein}
For each sign vector $\epsilon$, we have $E^{\epsilon}(N,2-k,\chi)=\text{span}_\mathbb{C}\{E^{\epsilon}\}$.
\end{Lem}
\begin{proof}
We first note that each $E^\epsilon(N,2-k,\chi)$ has dimension $1$ by Lemma 4.4 of \cite{zhang2014isomorphism}. It suffices to show that $E^\epsilon$ or $\sum_m\epsilon_m E_m$ satisfies the $\epsilon$-condition.

Let $n$ be any positive integer such that for some $p\mid N$ we have $\chi_p(n)=-\epsilon_p$. We need to show that the $q^n$ coefficient of $\sum_m\epsilon_mE_m$ is $0$. Indeed, such a coefficient is
\[2\sum_{m\mid N}\epsilon_m\sum_{d\mid n}\chi_m(n/d)\chi_m'(d)d^{1-k}=2\sum_{m\mid\frac{N}{N_p}}\epsilon_m\sum_{d\mid n}d^{1-k}\left(\epsilon_p\chi_{pm}(n/d)\chi_{pm}'(d)+\chi_m(n/d)\chi_m'(d)\right)=0,\]
since $\epsilon_p\chi_{pm}(n/d)\chi_{pm}'(d)=\epsilon_p\chi_{p}(n)\chi_{m}(n/d)\chi_{m}'(d)$.
\end{proof}

For an integer $m$, we say that $m$ is an $\epsilon$-integer if $\chi_p(m)\neq -\epsilon_p$ for each $p\mid N$. For any $P=\sum_{n}a(n)q^n\in\mathbb C[q^{-1}]$, we say $P$ is an $\epsilon$-polynomial in $q^{-1}$ if $n$ is an $\epsilon$-integer whenever $a(n)\neq 0$.

We state the obstruction theorem for scalar-valued modular forms. (See Theorem 6 in \cite{bruinier2003borcherds} in the case of prime level.) Here we remark that the case $k\geq 2$ is trivial, since $2-k\leq 0$ and the obstruction space is trivial.

\begin{Thm}\label{Obstruction-Scalar}
Let $k\leq 0$ and $\epsilon$ be a sign vector. Let $P=\sum_{n<0}a(n)q^n$ be an $\epsilon$-polynomial in $q^{-1}$. Then there exists $f\in A^\epsilon(N,k,\chi)$ with $f=\sum_{n\in\mathbb Z}a(n)q^n$, if and only if
\[\sum_{n<0}s(n)a(n)b(-n)=0,\]
for each $g=\sum_{n\geq 0}b(n)q^n\in S^{\epsilon^*}(N,2-k,\chi)$. If $f$ exists, it is unique and its constant term is given by
\[a(0)=-\sum_{n< 0}s(n)a(n)B(-n).\]
\end{Thm}
\begin{proof} The statements follow from Proposition 4.1, Remark 4.2 and Lemma 4.3. For details, please see the proof of Theorem 4.5 in \cite{zhang2014isomorphism}.
\end{proof}

For completeness, we reproduce a couple of results in \cite{zhang2014isomorphism} following the lines in \cite{bruinier2003borcherds}. These results concern the rationality of the Fourier coefficients, which is important in Borcherds's theory of automorphic products. For $f=\sum_n a(n)q^n$ and $\sigma\in\text{Gal}(\mathbb C/\mathbb Q)$, define $f^\sigma=\sum_n a(n)^\sigma q^n$. Let $k$ be an integer, possibly positive.

\begin{Lem}\label{Galois}
If $f\in A(N,k,\chi)$, so is $f^\sigma$.
\end{Lem}
\begin{proof}
This is Lemma 4.6 in \cite{zhang2014isomorphism}.
\end{proof}

The following generalizes Proposition 4.7 in \cite{zhang2014isomorphism} to our present more general setting.

\begin{Prop}
Let $k\leq 0$ be an integer and fix any sign vector $\epsilon$. Let $f=\sum_na(n)q^n\in A^\epsilon(N,k,\chi)$ and suppose that $a(n)\in\mathbb Q$ for $n<0$. Then all coefficients $a(n)$ are rational with bounded denominator.
\end{Prop}
\begin{proof}
Let $\sigma\in\text{Gal}(\mathbb C/\mathbb Q)$. We note that $f^\sigma\in A^\epsilon(N,k,\chi)$. Indeed, from Lemma \ref{Galois}, we see that $f^\sigma\in A(N,k,\chi)$; moreover, the Galois action preserves the $\epsilon$-condition.

Now consider $h=f-f^\sigma\in A^\epsilon(N,k,\chi)$. It is obvious that $h$ is holomorphic at $\infty$, hence $h\in M(N,k,\chi)$ by Corollary \ref{Holomorphy}. But $k\leq 0$, so $M(N,k,\chi)=\{0\}$. It follows that $f$ has rational coefficients. Since $f\Delta^{k'}\in M(N,k+12k',\chi)$ for large $k'$, it has coefficients with bounded denominator, hence so does $f$.
\end{proof}

\section{Zagier Duality}

\noindent
In this section, we prove the Zagier duality for integral weight weakly holomorphic modular forms.

We assume that $N_2$, the $2$-part of $N$, is $1$ or $4$, for simplicity.  Actually this condition on $N$ will not be used until Theorem 5.7 on Zagier duality. We keep other notations in the previous section. To better describe the statements, we introduce a notion of weakly holomorphic modular forms with $\epsilon$-condition.

\begin{Def}
Let $k,m$ be integers and $\epsilon$ any sign vector. We call $f=\sum_{n}a(n)q^n\in A^\epsilon(N,k,\chi)$ \emph{reduced of order} $m$, if $f=\frac{1}{s(m)}q^{m}+O(q^{m+1})$ and for each $n>m$ with $a(n)\neq 0$, there does not exist $g\in A^\epsilon(N,k,\chi)$ such that $g=q^n+O(q^{n+1})$.
\end{Def}

\begin{Lem}
For any integer $m$, there exists at most one $f\in A^\epsilon(N,k,\chi)$ such that $f$ is reduced of order $m$.
\end{Lem}
\begin{proof}
Suppose there are two such modular forms, say $f=\sum_na(n)q^n$ and $g=\sum_nb(n)q^n$, we prove that they must be equal. Indeed, if $f\neq g$, then let $n$ be the smallest such that $a(n)\neq b(n)$. Now $f-g\in A^\epsilon(N,k,\chi)$ and its Fourier expansion begins with the $q^n$-term. Since at least one of $a(n)$ and $b(n)$ is not $0$, this contradicts to the assumption that both $f$ and $g$ are reduced.
\end{proof}

If the data $N,k,\chi$ and $\epsilon$ are clear from the context, we shall denote by $f_m$ the reduced modular form of order $m$ in $A^\epsilon(N,k,\chi)$ if it exists. We shall also denote by $f_m'$ the reduce modular form of order $m$ in $A^{\epsilon^*}(N,2-k,\chi)$ if it exists.

For each integer $m$, let $A^\epsilon_m(N,k,\chi)$ denote the subspace of modular forms $f\in A^\epsilon(N,k,\chi)$ such that $f=\sum_{n\geq m}a(n)q^n$. For example, $A^\epsilon_0(N,k,\chi)=M^\epsilon(N,k,\chi)$ and $A^\epsilon_1(N,k,\chi)=S^\epsilon(N,k,\chi)$ by Corollary 3.6.

\begin{Prop} For any integer $m$,
the set $\{f_n: n\geq m, f_n\text{\ exists}\}$ is a basis for $A^\epsilon_m(N,k,\chi)$. In particular, we have canonical  bases for the spaces $A^\epsilon(N,k,\chi)$, $M^\epsilon(N,k,\chi)$ and $S^\epsilon(N,k,\chi)$ that consist of reduced modular forms.
\end{Prop}
\begin{proof}
Clearly such a set is linearly independent since they have different lowest power terms.
We then only need to show that reduced forms span the whole space.
Let $f$ be any non-zero form in $A^\epsilon(N,k,\chi)$ and we may assume that $f=\frac{1}{s(m)}q^m+O(q^{m+1})$. If $m>0$, then $f$ is a cusp form. We know that $m$ cannot be big, since otherwise it forces $f=0$ by Sturm's Theorem (see \cite{sturm1987congruence} or Theorem 6.4 below); note that such Sturm's bound depends only on $N$.

If $f$ is reduced, then we are done. If not, we have a finite set of integers $n$ such that $n>m$ and a modular form $g_n=q^n+O(q^{n+1})\in A^\epsilon(N,k,\chi)$ exists. By induction on $m$, we may assume that $g_n$ is a linear combination of reduced modular forms for each $n$. It is clear that $f-\sum_nc_ng_n$ is reduced for some scalars $c_n$, hence $f$ itself is a linear combination of reduced modular forms.
\end{proof}

For an integer $m$, we say that $m$ satisfies the $\epsilon$-condition, if $\chi_p(m)\neq -\epsilon_p$ for each $p\mid N$.

\begin{Lem}
Assume $k\geq 2$ and $m\leq 0$. Then there exists a reduced modular form of order $m$ in $A^\epsilon(N,k,\chi)$ if and only if $m$ is an $\epsilon$-integer. If it exists, such a form contain precisely one non-positive power term, that is $\frac{1}{s(m)}q^m$.
\end{Lem}
\begin{proof}
By Proposition 4.1 and our isomorphism, we can have the obstruction theorem for scalar-valued modular forms in the case when $k\geq 2$. Here $2-k\leq 0$, hence the obstruction space is trivial. Therefore, any $\epsilon$-polynomial in $q^{-1}$ lifts to a modular form in $A^\epsilon(N,k,\chi)$. In particular, all $\epsilon$-monomials in $q^{-1}$ can be lifted. Then for a reduced modular form, there is only one non-positive power term in $q$.
\end{proof}

\begin{Lem} Assume $k\leq 0$ and $m<0$. Let $\{f_n': n\in S\}$ be the basis of reduced modular forms for $S^{\epsilon*}(N,2-k,\chi)$; here $S$ is a uniquely determined finite set of positive integers.
Then $f_m$ exists if and only if $m$ is an $\epsilon$-integer and $-m\not\in S$.
\end{Lem}
\begin{proof}
Assume that $S=\{n_1,n_2,\cdots, n_k\}$ with $n_i<n_{i+1}$ for $1\leq i<k$, and $f_{n_i}'=\sum_{n}a_i(n)q^n$ for each $i$.

It is clear that if $-m=n_i$, then $f_m$ does not exist because of the obstruction by $f_{n_i}'$. Conversely, suppose $-m\not\in S$ and consider the polynomial
\[P=\frac{1}{s(m)}q^{m}-\frac{1}{s(m)}\sum_{i}s(n_i)a_i(m)q^{-n_i}.\]
Since $f_{n_i}'$ are reduced, we must have $s(n_i)a_j(n_i)=\delta_{i,j}$. From this, we see that $P$ satisfies the obstruction conditions. Moreover, since $n_i$ satisfy the $\epsilon^*$-condition, $-n_i$ satisfies the $\epsilon$-condition and $P$ is an $\epsilon$-polynomial in $q^{-1}$. By Theorem 4.4, there exist a modular form with $P$ the principal part, and the existence of $f_m$ follows.
\end{proof}

We shall assume $k\neq 1$. Because of the dual weights $k$ and $2-k$, without loss of generality, from now on until the end of this section, we assume that $k\leq 0$.

\begin{Lem}
Let $\epsilon=(\epsilon_p)$ be any sign vector and let $\epsilon^*=(\epsilon_p^*)$ with $\epsilon^*_p=\chi_p(-1)\epsilon_p$. Assume $m,d\in\mathbb Z$, $m<0$, and that both of the reduced modular forms \[f_{m}=\sum_{n\in\mathbb Z}a_{m}(n)q^n\in A^\epsilon(N,k,\chi) \quad \text{\ and\ }\quad  f_{d}'=\sum_{n\in\mathbb Z}b_{d}(n)q^n\in A^{\epsilon^*}(N,2-k,\chi)\]exist. Then $a_m(-n)b_d(n)=0$ for any $d<n<-m$.
\end{Lem}
\begin{proof}
Assume first that $d\leq 0$. Since $2-k\geq 2$ and $f_d'$ is reduced, we must have $b_d(n)=0$ if $-d<n\leq 0$ by Lemma 5.4. Therefore, we only need to consider the case when $0<n<-m$.
We fix any $0<n_0<-m$ such that $a_m(-n_0)\neq 0$ and it suffices to prove that $b_d(n_0)=0$. We first note that $f_{-n_0}$ does not exist and by Lemma 5.5 this means that $f_{n_0}'$ exists in the dual cusp form space. Since $f_d'$ is also reduced, we must have $b_d(n_0)=0$.

Similarly, if $d>0$, then for any $d<n_0<-m$ with $a_m(n_0)\neq 0$, we must have the existence of $f_{n_0}'$. That $f_d'$ is reduced implies $b_d(n_0)=0$.
\end{proof}

\begin{Thm}\label{MainTheorem} Let $\epsilon=(\epsilon_p)$ be any sign vector and let $\epsilon^*=(\epsilon_p^*)$ with $\epsilon^*_p=\chi_p(-1)\epsilon_p$. Assume $m,d\in\mathbb Z$ with $m<0$. Assume that both of the reduced modular forms \[f_{m}=\sum_{n\in\mathbb Z}a_{m}(n)q^n\in A^\epsilon(N,k,\chi) \quad \text{\ and\ }\quad  f_{d}'=\sum_{n\in\mathbb Z}b_{d}(n)q^n\in A^{\epsilon^*}(N,2-k,\chi)\] exists. Then we  have $a_{m}(-d)=-b_{d}(-m)$.
\end{Thm}
\begin{proof}
We denote $f=f_{m}$ and $f'=f_{d}'$. Since $ff'\in A(N,2,1)$, we have a meromorphic $1$-form $ff'd\tau$ on the compact Riemann surface $X_0(N)$, so the sum of residues of $ff'd\tau$ must vanish. Since $ff'$ is holomorphic on $\mathbb H$, $ff'd\tau$ is holomorphic on $X_0(N)$ except at the cusps. For a cusp $s$ with width $h_s$, let $\alpha\in\text{SL}_2(\mathbb Z)$ such that $\alpha\infty\sim s$. Then we know that the residue of $ff'd\tau$ at $s$ is given by
the constant term in $q_s$ of $\frac{h_s}{2\pi i}(ff')|\alpha$. Here $q_s=q^{1/h_s}$. The reader may see Section 2.3 of Miyake's book \cite{miyake2006modular} for more details on this. Let $m_1$ be an odd positive divisor of $N$.

We first deal with the case when $N$ is odd. For a cusp $s\sim \frac{1}{N/m_1}$, $\gamma_{m_1}\infty\sim s$ and then it is easy to see that the residue of $ff'd\tau$ at $s$ is given by the constant term of $\frac{1}{2\pi i}(f|\eta_{m_1})(f'|\eta_{m_1})$.

Since $f\in A^\epsilon(N,k,\chi)$, by Lemma 2.5 and Lemma 3.1, we obtain
\[f|\eta_{m_1}= \left(\prod_{p\mid m_1} \epsilon_{p}\varepsilon_{p}\chi_p'(p)\right)m_1^{\frac{1-k}{2}}f|U(m_1).\] Similarly
\[f'|\eta_{m_1}=\left(\prod_{p\mid m_1} \epsilon^*_{p}\varepsilon_{p}\chi_p'(p)\right)m_1^{\frac{k-1}{2}}f'|U(m_1).\]
Hence $(f|\eta_{m_1})(f'|\eta_{m_1})=(f|U(m_1))(f'|U(m_1))$.

For ease of notations, for $a,b\in\mathbb Z$, we define
\[c(a,b)=\left\{
\begin{array}{cl} 0 & \text{ if\ \ } a\nmid b,\\
\frac{1}{s(b)}& \text{ if\ \ } a\mid b.
\end{array}\right.\]
Then the constant term of $(f|U(m_1))(f'|U(m_1))$ is given by
\[c(m_1,m)b_{d}(-m)+c(m_1,d)a_m(-d),\] where other terms vanish by Lemma 5.6. Summing over all $m_1\mid N$, we have
\[\sum_{m_1\mid N}c(m_1,m)=\sum_{m_1\mid N}c(m_1,d)=1.\]
Since the inequivalent cusps are precisely represented by $\frac{1}{N/m_1}$, the sum of all all residues of $ff'd\tau$ is given by $\frac{1}{2\pi i}(a_m(-d)+b_{d}(-m))$. It follows that $a_m(-d)=-b_{d}(-m)$ and we are done with this case.

Now let us assume $2^2||N$. The idea is the same, but the computations in this case are more complicated. For each $m_1\mid \frac{N}{4}$, we have the same expression as above for $f|\eta_{m_1}$ and for $f'|\eta_{m_1}$, and from the same argument we see that at the cusp $s\sim\frac{1}{N/m_1}$, the residue of $ff'd\tau$ at $s$ is given by the constant term of $\frac{1}{2\pi i}(f|\eta_{m_1})(f'|\eta_{m_1})=\frac{1}{2\pi i}(f|U(m_1))(f'|U(m_1))$. This is
\[\frac{1}{2\pi i}\left(c(m_1,m)b_{d}(-m)+c(m_1,d)a_m(-d)\right).\]

Similarly, at the cusp $s\sim \frac{1}{N/4m_1}$,  the residue of $ff'd\tau$ at $s$ is given by the constant term of $\frac{1}{2\pi i}(f|\eta_{2m_1})(f'|\eta_{2m_1})=\frac{1}{2\pi i}(f|U(4m_1))(f'|U(4m_1))$. This is
\[\frac{1}{2\pi i}\left(c(4m_1,m)b_{d}(-m)+c(4m_1,d)a_m(-d)\right).\]

We are left with the cusps of the form $s\sim \frac{1}{2m_1}$ with $m_1\mid\frac{N}{4}$. This time the matrix $\alpha_{2m_1}$ in Remark 3.5 will do the job. From the expression there, we see that
the residue of $ff'd\tau$ at $s$ is given by the constant term of $\frac{1}{2\pi i}(f|\alpha_{2m_1})(f'|\alpha_{2m_1})$, which is given by
\[\frac{1}{2\pi i}\left(\chi_2(m/2)^2c(N/2m_1,m)b_{d}(-m)+\chi_2(d/2)^2c(N/2m_1,d)a_{m}(-d)\right).\] Note here that if $2\nmid m$, the value of the first term is understood to be $0$ even though $\chi_2(m/2)$ is not defined; the $c$-factor is $0$ anyway. The same interpretation applies to the second term.

By elementary computations, we see that
\begin{itemize}
\item $\sum_{m_1\mid N_1}c(4m_1,m)=1/2$ if $4\mid m$ and $0$ otherwise,
\item $\sum_{m_1\mid N_1}c(m_1,m)=1/2$ if $2\mid m$ and $1$ otherwise,
\item $\sum_{m_1\mid N_1}\chi_2(m/2)^2c(N/2m_1,m)=1/2$ if $2\mid\mid m$ and $0$ otherwise,
\end{itemize}
and they also hold with $m$ replaced by $d$. Summing over $m_1\mid \frac{N}{4}$, we see that in this case we also have $a_m(-d)=-b_{d}(-m)$. This finishes the proof.
\end{proof}

\begin{Rmk}
When $N=5,13,17$, this is due to Rouse \cite{rouse2006zagier} and to Choi \cite{choi2006simple} (note that in this paper the factor $\varepsilon_p$ is missing in and after Lemma 1.5). Note that we not only generalize their results to a more general $N$, but also remove the dependence on existence and extend the duality to include holomorphic forms. Moreover, if $f_m$ exists and $a_m(-d)\neq 0$ with $-d>m$, then by Lemma 5.5, it is easy to see that $f_{d}'$ exists. Conversely, if $f_d'$ exists and $b_d(-m)\neq 0$ with $-m>d$, then by Lemma 5.5, Lemma 4.3 and Theorem 4.4, it is easy to see that $f_{m}$ exists.  Hence, we have the complete grids for Zagier duality, in the sense that all nonzero non-leading coefficients of reduced modular forms are covered by the duality in Theorem 5.7.
\end{Rmk}

We finish this section with a few examples.  In \cite{kim2013weakly}, for a special $\epsilon$, cases when the obstruction space is trivial were treated, namely when $N=8,12$ or $21$. To best illustrate the theory, here we consider the case when $N=15$.

We know that $\chi=\left(\frac{\cdot}{15}\right)$, and $\chi_3=\left(\frac{\cdot}{3}\right)$, $\chi_5=\left(\frac{\cdot}{5}\right)$. There are four distinct sign vectors $\epsilon$:
\[\epsilon_1=(-1,-1),\quad \epsilon_2=(1,-1),\quad \epsilon_3=(-1,1),\quad \epsilon_4=(1,1).\]
Among them, $\epsilon_1$ and $\epsilon_2$ are dual to each other, and $\epsilon_3$ and $\epsilon_4$ are dual to each other.

Since in this case the signature $r$ of all possible discriminant forms $3^{\pm 1}\oplus 5^{\pm 1}$ satisfies $\frac{r}{2}\equiv \frac{(3-1)+(5-1)}{2}\equiv 1\imod 2$, we should consider odd weights instead. For simplicity, let us consider the case $k=-1$ and $2-k=3$. We consider examples for two $\epsilon$ in a moment, one of which has trivial obstructions while the other does not.

We first look at the data for weight $3$ homomorphic modular forms. We know that $S(15,3,\chi)=\mathbb Cg_1+\mathbb Cg_2$, with
\begin{align*}g_1&=q - 3q^4 - 3q^6 + 9q^9 + 5q^{10} +O(q^{15})\in S^{\epsilon_4}(15,3,\chi),\\
g_2&=q^2 - 3q^3 + 5q^5 - 7q^8 + 9q^{12}+O(q^{15})\in S^{\epsilon_1}(15,3,\chi).
\end{align*} The Eisenstein space $E(15,3,\chi)=\sum_i \mathbb C E^{\epsilon_i}$ with
\begin{align*}E^{\epsilon_1}&=\frac{1}{4} - \frac{5}{8}q^2 - \frac{5}{8}q^3 - \frac{13}{8}q^5 - \frac{85}{8}q^8 - \frac{105}{8}q^{12}+O(q^{15}),\\
E^{\epsilon_2}&=\frac{1}{4} + \frac{1}{2}q^3 + 6q^7 + \frac{15}{2}q^{10} +
\frac{21}{2}q^{12} + 21q^{13}+O(q^{15}),\\
E^{\epsilon_3}&= \frac{1}{4} + \frac{3}{2}q^5 + \frac{5}{2}q^6 + 5q^9 + 15q^{11} + 30q^{14}+O(q^{15}),\\
E^{\epsilon_4}&= \frac{1}{4} - \frac{1}{8}q -\frac{21}{8}q^4 - \frac{25}{8}q^{6} - \frac{41}{8}q^{9} - \frac{65}{8}q^{10}+O(q^{15}).
\end{align*}
We note that $E^{\epsilon_1}$ and $E^{\epsilon_4}$ are not reduced and their Fourier coefficients have big denominators, because of the existence of $g_1,g_2$. Such integrality will be considered in Section 6.

\begin{Exa} Let $\epsilon=\epsilon_1$ hence $\epsilon^*=\epsilon_2$. From above data, we see that there exist no obstructions for $A^\epsilon(15,-1,\chi)$. From the $\epsilon$-condition, we see that $f_m$ exists if and only if $m\equiv 0, 2,3,5,8,12\imod 15$. The basis of reduced modular forms for $A^\epsilon(15,-1,\chi)$ starts with:
\begin{align*}f_{-3}&=\frac{1}{2}q^{-3} -\frac{1}{2} + 3q^2 - \frac{1}{2}q^3 - 3q^5 - 3q^8 + 6q^{12}+O(q^{15}),\\
f_{-7}&=q^{-7} - 6 + 12q^2 + 33q^3 + 39q^5 - 140q^8 - 144q^{12}+O(q^{15}),\\
&\vdots
\end{align*}
On the other hand, the basis of reduced forms for $A^{\epsilon^*}(15,3,\chi)$ begins with $f_0'=E^{\epsilon_2}$:
\begin{align*}f_{0}'&=\frac{1}{4} + \frac{1}{2}q^3 + 6q^7 + \frac{15}{2}q^{10} +
\frac{21}{2}q^{12} + 21q^{13}+O(q^{15}),\\
f_{-2}'&=q^{-2} - 3q^3 - 12q^7 - 45q^{10} + 36q^{12} + 146q^{13} +O(q^{15}),\\
&\vdots
\end{align*}
We can easily detect the Zagier duality for these basis elements, by ignoring the first columns in these two tables and viewing one table horizontally and the other vertically.
\end{Exa}

\begin{Exa}
Now let $\epsilon=\epsilon_3$ and $\epsilon^*=\epsilon_4$. Because of the existence of $g_1$, there is a non-trivial obstruction condition for $A^\epsilon(15,-1,\chi)$. The if $f_m$ exists, the $\epsilon$ condition says $m\equiv 0, 5,6,9,11,14\imod 15$. By Lemma 5.5, we see that $f_m$ exists if and only if $m\neq -1$ and $m\equiv 0,5,6,9,11,14\imod 15$.

The basis of reduced modular forms for $A^\epsilon(15,-1,\chi)$ starts with
\begin{align*}
f_{-4}&=q^{-4} + 3q^{-1} + 3 - 7q^5 + 3q^6 - 21q^9 - 11q^{11} + 44q^{14} +O(q^{15}),\\
f_{-6}&=\frac{1}{2}q^{-6} + 3q^{-1} + \frac{7}{2} + 21q^5 - \frac{49}{2}q^6 + 19q^9 - 147q^{11} +
99q^{14} + O(q^{15}),\\
f_{-9}&=\frac{1}{2}q^{-9} - 9q^{-1} + 4+ 99q^5+48q^6   -275q^9+360q^{11}  -2160q^{14} + O(q^{15}),\\
&\vdots&
\end{align*}
On the other hand, the basis of reduced modular forms for $A^{\epsilon^*}(15,3,\chi)$ starts with
\begin{align*}
f_{1}'&=q - 3q^4 - 3q^6 + 9q^9 + 5q^{10} +O(q^{15}),\\
f_{0}'&=\frac{1}{4} -3q^4 - \frac{7}{2}q^{6} - 4q^{9} - \frac{15}{2}q^{10}+O(q^{15}),\\
f_{-5}'&=\frac{1}{2}q^{-5} + 7q^4 - 21q^6 - 99q^9 + 67q^{10} +O(q^{15}),\\
&\vdots&
\end{align*}
Here $f_1'=g_1$ and $f_0'$ can be obtained by $E^{\epsilon_4}+\frac{1}{8}g_1$. The duality is also clear from these two tables.
\end{Exa}

\section{Integrality of Reduced Modular Forms}

\noindent
Borcherds \cite{borcherds1999gross} raised the question on the existence of a basis with integral Fourier coefficients for vector-valued modular forms associated with Weil representations, to which McGraw \cite{mcgraw2003rationality} gave an affirmative answer.
In this section, we consider a related but different problem, that is, the integrality for Fourier coefficients of reduced modular forms. 

Let $k$ be an integer such that $k\neq 1$. In this section, we consider the following question:
\begin{Que*}
For any reduced modular form $f_m=\sum_na(n)q^n$, do we always have $s(n)a(n)\in\mathbb Z$?
\end{Que*}

Such integrality first appeared in \cite{kim2012rank} and then in \cite{kim2013weakly}, where such integrality is crucial since $s(n)a(n)$ represents the multiplicity of some root in a generalized Kac-Moody superalgebra.  Such a question concerns the existence of a \emph{Miller-like} basis for weakly holomorphic modular forms with $\epsilon$-condition, and because of the isomorphism, it is essentially the integrality of the corresponding basis for vector-valued modular forms. Although numerical evidence indicates an affirmative answer to the above question, at present time, we do not know how to prove this systematically for all $N$. In the rest of this section, we will illustrate how to computationally prove it for any fixed $N$.

Recall that we denote by $f_m$ the reduced modular form of order $m$, when it exists and the data $N,k,\epsilon$ are clear in the context, and by $f'_m$ the reduced modular form in $A^{\epsilon^*}(N,2-k,\chi)$. We write $f_m=\sum_na_m(n)q^n$. Let $m_{\epsilon}$ denote the maximal $m$ such that $f_m'$ exists. If we would like to emphasize the sign vector, we shall write $f_{m,\epsilon}$, $a_{m,\epsilon}(n)$, and $f_{m,\epsilon}'$ accordingly.

The following lemma reduces the question to testing a finite number of reduced modular forms.
\begin{Lem}
Assume that for all $n\in\mathbb Z$ and $m\geq -N-m_\epsilon$, we have $s(n)a_m(n)\in \mathbb Z$. Then $s(n)a_m(n)\in\mathbb Z$, for all $m,n\in\mathbb Z$.
\end{Lem}
\begin{proof}
Consider any reduced modular form $f_{m'}$ with $m'<-N-m_\epsilon$. There exists integers $-N-m_\epsilon\leq m_0'<m_\epsilon$ and $l>1$ such that $m'=-Nl+m_0'$. The existence of $f_{m'}$ implies that of $f_{m_0'}$ by Lemma 5.4 and Lemma 5.5. Consider now
\[g=j(N\tau)^{l}f_{m_0'}=\sum_nb(n)q^n\in A^\epsilon(N,k,\chi).\]
Here $j=q^{-1}+744+O(q)$ is the weight $0$ modular form of level $1$. Since $j$ has integral Fourier coefficients, by the assumption on $f_{m_0'}$, we see that $b(n)s(n)\in\mathbb Z$ for each $n$.

Now $g$ and $f_{m'}$ share the same lowest power term, and we must have that
\[f_{m'}=g-\sum_{m>m'} s(m)b(m)f_m.\]
Hence $s(n)a_{m'}(n)=s(n)b(n)-s(m)b(m)s(n)a_{m}(n)\in\mathbb Z$ by the assumption and induction on $m$. We are done.
\end{proof}

To consider the integrality of a fixed reduced modular form, Sturm's Theorem (\cite{sturm1987congruence}, see also \cite[Corollary 3.2]{kim2013weakly}) will be useful. In \cite{kim2013weakly}, we applied a trick to the weight $0$ reduced forms $f_{-1}$ where only the constant term is possibly half-integral, and the fact that constant functions are modular forms for $\Gamma_1(N)$ is important. We end this section with the following example, which, in particular, deals with the cases $f_{-m}$ when $(m,N)>1$.

\begin{Exa} Let us treat the simplest case $N=3$ and $k=-1$. Firstly, we have $\chi=\left(\frac{\cdot}{3}\right)$, $\epsilon_1=+1$ and $\epsilon_2=-1$. Since $S(3,3,\chi)=\{0\}$, no obstructions exist for $A^\epsilon(3,-1,\chi)$ for each $\epsilon$. The bound in Lemma 6.3 is now $-3$.
To establish the integrality for all reduced modular forms, we only have to consider the integrality of $f_{-2,\epsilon_1}$ and $f_{-3, \epsilon_1}$ for $A^{\epsilon_1}(3,0,\chi)$, and
$f_{-1,\epsilon_2}$ and $f_{-3, \epsilon_2}$ for $A^{\epsilon_2}(3,0,\chi)$.
Explicitly the first few term of these modular forms are
\begin{align*}
f_{-1,\epsilon_2}&=q^{-1} + 9 - 82q^2 + 189q^3 - 892q^5 + 1782q^6 - 6234q^8 +O(q^{9}),\\
f_{-2,\epsilon_1}&=q^{-2} - 27 + 328q - 7128q^3 + 24854q^4 - 221859q^6 + 591632q^7 +O(q^{9}),\\
f_{-3,\epsilon_1}&=\frac{1}{2}q^{-3} - 36 - 1701q - 50058q^3 - 499608q^4 - 4023392q^6 -
27788508q^7+O(q^{9}),\\
f_{-3,\epsilon_2}&=\frac{1}{2}q^{-3} + 45 + 16038q^2 + 50058q^3 + 2125035q^5 + 4023310q^6 +89099838q^8+O(q^{9}).
\end{align*}
The coefficients for $f_{-3,\epsilon_1}$ and $f_{-3,\epsilon_2}$ are not integral for some large powers, for example their coefficients for $q^{45}$ are both half integral. The computation of these Fourier expansions involves the following $\eta$-quotients:
$H_1=\eta(\tau)^{-3}\eta(3\tau)^9$ and $H_2=\eta(\tau)^9\eta(3\tau)^{-3}$.
The integrality of $f_{-1,\epsilon_2}$ follows from the fact that $H_1f_{-1,\epsilon_2}\in M(3,2,1)$ and the Sturm bound is then $\frac{2}{3}$. The integrality of $f_{-2,\epsilon_1}$ follows from the fact that $H_1^2f_{-2,\epsilon_1}\in M(3,5,\chi)$ and the Sturm bound is then $\frac{5\cdot 8}{12}<4$.

To deal with the rest of two reduced modular forms, note first that for $i=1,2$, $H_1^3H_2f_{-3,\epsilon_i}\in M(3,11,\chi)$.
The Sturm bound in both cases is $\frac{11\cdot 8}{12}<8$. By Corollary 3.2 in \cite{kim2013weakly}, we see that all of the three modular forms
\[2f_{-3,\epsilon_1}, \quad 2f_{-3,\epsilon_2}\quad f_{-3,\epsilon_1}+f_{-3,\epsilon_2}\]
have integral coefficients. On the one hand, if $3\mid n$, then $s(n)a_{-3,\epsilon_i}(n)\in\mathbb Z$ for $i=1,2$. On the other hand, if $3\nmid n$, then one of $a_{-3,\epsilon_1}(n)$ and $a_{-3,\epsilon_2}(n)$ is zero. But the sum of these two coefficient is integral, so both of them are integral.
\end{Exa}

\vskip 0.5 cm

\addcontentsline{toc}{chapter}{Bibliography}
\bibliographystyle{amsplain}
\bibliography{paper}

\end{document}